\documentclass[11pt,letterpaper]{amsart}
\usepackage[T1]{fontenc}
\usepackage{lmodern}
\usepackage{amsmath,amssymb,amsthm,mathtools}
\usepackage{microtype}
\usepackage{needspace}
\usepackage[hidelinks]{hyperref}
\hypersetup{pdftitle={Nondeformability of Moishezon Manifolds with Nonnegative Kodaira Dimension},pdfauthor={Mu-Lin Li}}
\numberwithin{equation}{section}
\newtheorem{theorem}{Theorem}[section]
\newtheorem{proposition}[theorem]{Proposition}
\newtheorem{lemma}[theorem]{Lemma}

\theoremstyle{definition}

\theoremstyle{remark}

\newtheorem{example}[theorem]{Example}
\newtheorem{quest}[theorem]{Question}
\newcommand{\X}{\mathcal X}
\newcommand{\Y}{\mathcal Y}
\newcommand{\D}{\Delta}
\newcommand{\Ds}{\Delta^*}
\newcommand{\Exc}{\operatorname{Exc}}

\newcommand{\codim}{\operatorname{codim}}

\newcommand{\Cyc}{\mathcal C}

\title[Nondeformability of Moishezon  manifolds]{Nondeformability of Moishezon manifolds with nonnegative Kodaira dimension}
\author{Mu-Lin Li}
\address{School of Mathematics, Hunan University, China}
\email{mulin@hnu.edu.cn}
\thanks{The author is supported by NSFC (Nos.~12271073 and 12271412).}
\keywords{Global nondeformability, isotrivial family, Moishezon morphism, relative cycle space, Kodaira dimension}
\subjclass[2020]{14D15, 14E30, 14J10, 32G05}

\date{}

\begin{document}

\begin{abstract}
Let $\pi:\X\to\D$ be a proper holomorphic submersion with connected
fibers, and suppose that $X_t$ is biholomorphic to a fixed smooth
Moishezon  manifold $S$ for every $t\ne0$. If $\kappa(S)\ge0$, we prove
that $X_0\simeq S$.
\end{abstract}

\maketitle

\section{Introduction}

Let $\D\subset\mathbb C$ be a disk centered at $0$, and write
$\Ds=\D\setminus\{0\}$. Throughout, a smooth family over $\D$ means
a proper holomorphic submersion with connected compact fibers. We
consider a smooth family
\[
\pi:\X\longrightarrow\D,\qquad X_t=\pi^{-1}(t),
\]
whose fibers have complex dimension $n\ge1$.

Siu proposed several interesting problems recently. One of them is the following global nondeformability.
\begin{quest}
Given a smooth family $\pi:\X\to \Delta$ such that $X_t$ is biholomorphic to a given compact complex manifold $S$ for $t\neq0$, the global nondeformability problem is to show that $X_0$ must also be biholomorphic to $S$.
\end{quest}
Affirmative results are known in several geometric settings.
Wright's work \cite{Wr77} yields global nondeformability for compact
hyperbolic manifolds. More recently, the author--Rao--Wang \cite{lrw} proved
that, in a smooth family over a disk, the locus of fibers isomorphic
to a fixed compact hyperbolic manifold is either discrete or the
entire disk. On the rational side, Siu \cite{Siu0} established global
nondeformability for complex projective space, and Hwang \cite{Hw}
proved the corresponding result for smooth hyperquadrics of dimension
at least three. Using the theory of varieties of minimal rational
tangents, Hwang--Mok \cite{HN1,HN2,HN3} established rigidity under
K\"ahler deformations for broad classes of rational homogeneous
manifolds of Picard number one. For projective manifolds with
semiample canonical bundle, the author--Liu \cite{LL26} proved global
nondeformability under a K\"ahler morphism hypothesis, and
the author--Rao--Wang \cite{LRW} obtained a stronger statement concerning the
locus of isomorphic fibers in a smooth K\"ahler family. The author \cite[Corollary~1.7]{Li26}
subsequently proved global nondeformability for non-uniruled compact
K\"ahler manifolds in families whose fibers are all K\"ahler.

The main result of this paper establishes global nondeformability
for smooth Moishezon manifolds of nonnegative Kodaira dimension.

\begin{theorem}\label{thm:main}Let $\pi:\X\to\D$ be a smooth family.
Suppose that $S$ is a smooth connected Moishezon manifold such that
\[
\kappa(S)\ge0,\qquad X_t\simeq S\quad(t\in\Ds).
\]
Then $X_0\simeq S$.
\end{theorem}

\begin{example}\label{ex:hirzebruch}\cite[(8.1)Theorem (iv)]{BPV}
There is a smooth projective family over a disk whose noncentral
fibers are $\mathbb P^1\times\mathbb P^1$ and whose central fiber
is the Hirzebruch surface $\mathbb F_2$.
\end{example}

By \cite[Theorem~1.5]{Rao26}, the deformation limit of Moishezon manifolds is Moishezon. Thus every fiber
of $\pi$ is Moishezon. Koll\'ar's theorem
\cite[Corollary~22]{Kollar} then implies that $\pi$ is locally
Moishezon: after shrinking the base near any point, it is
bimeromorphic over the base to a projective morphism. The deformation limit of projective manifolds is Moishezon can also find in \cite[Theorem~1.5]{Li24}.  Consequently,
Theorem~\ref{thm:main} follows from the following relative statement.

\begin{theorem}\label{thm:local-moishezon}
Let $\pi:\X\to\D$ be a smooth family, locally Moishezon near $0$.
Suppose that $S$ is smooth connected Moishezon manifold, that
$\kappa(S)\ge0$, and that $X_t\simeq S$ for all $t\ne0$.
Then $X_0\simeq S$.
\end{theorem}

The proof of Theorem~\ref{thm:local-moishezon} has two steps.
First, Proposition~\ref{prop:criterion} treats a family equipped with
a bimeromorphic map over $\D$,
\[
F:\X\dashrightarrow S\times\D,
\]
which is biholomorphic over $\Ds$. The positivity of one plurigenus
of $X_0$ prevents its strict transform from being exceptional on the
product side. The resulting local product structure outside an
analytic subset of codimension at least two in $\X$ supplies a
horizontal vector field. Hartogs' theorem extends this field, and
its holomorphic flow trivializes the family near $0$.

Second, Proposition~\ref{prop:graph-selection} constructs such a
meromorphic product after a finite base change. We use the relative
Barlet space of cycles in $\X\times_\D(S\times\D)$ to parametrize
graphs of biholomorphisms $X_t\to S$. After restricting to a relatively
compact disk inside a disk over which the family is Moishezon,
Fujiki's properness result applies to its cycle-space components.
An analytic arc in a resolved parameter space then gives a finite
base change $\tau:\D_s\to\D$ and a bimeromorphic map
\[
F:\X\times_\D\D_s\dashrightarrow S\times\D_s
\]
that is biholomorphic over $\D_s^*$. Upper semicontinuity of a single
plurigenus completes the argument.

\subsection*{Notation and conventions}
For a smooth compact complex manifold $Z$ and an integer $m>0$, set
\[
P_m(Z)=h^0(Z,K_Z^{\otimes m}).
\]
Thus $\kappa(Z)\ge0$ is equivalent to $P_m(Z)>0$ for some $m>0$.
Meromorphic maps are understood in the analytic sense, with proper
graphs. For a proper modification $f$ between smooth manifolds,
$\Exc(f)$ denotes the locus in its source on which it is not an
isomorphism.

\section{Moishezon and modification preliminaries}\label{sec:prelim}

\subsection{Moishezon limits}

We need the following theorem on deformation limits of Moishezon manifolds.
\begin{theorem}[{\cite[Theorem~1.5]{Rao26}}]\label{thm:limit}
Let $f:M\to\D$ be a proper holomorphic submersion with connected
fibers. If $M_t$ is Moishezon for every $t\ne0$, then $M_0$ is
Moishezon.
\end{theorem}

The fiberwise conclusion is converted into a relative one by the
following theorem.

\begin{theorem}[{\cite[Corollary~22]{Kollar}}]
\label{thm:kollar}
A smooth proper morphism of normal irreducible complex analytic
spaces whose fibers are Moishezon is locally Moishezon.
\end{theorem}

In the proof of Theorem~\ref{thm:kollar}, the Hodge decomposition on
smooth Moishezon fibers yields the surjectivity required to make
$R^2f_*\mathcal O_M$ locally free; Koll\'ar's relative line-bundle
criterion then gives local Moishezonness. Thus
Theorem~\ref{thm:kollar} is precisely the bridge between the
fiberwise conclusion of Theorem~\ref{thm:limit} and the relative
cycle-space argument below.

\subsection{Pluricanonical forms and exceptional divisors}

\begin{lemma}\label{lem:birational-pluri}
If $f:W\to Z$ is a proper modification between smooth compact complex
manifolds, pullback induces an isomorphism
\[
H^0(Z,K_Z^{\otimes m})\xrightarrow{\sim}
H^0(W,K_W^{\otimes m})\qquad(m>0).
\]
Consequently, plurigenera are invariant under bimeromorphic maps
between smooth compact complex manifolds.
\end{lemma}

\begin{proof}
Pullback is injective because $f$ is an isomorphism on a dense open
set. Choose an analytic subset $C\subset Z$ of codimension at least
two such that $f$ is an isomorphism over $Z\setminus C$. A section
of $K_W^{\otimes m}$ determines a pluricanonical form on $Z\setminus C$.
Its coefficient in each local trivialization of $K_Z^{\otimes m}$ is
holomorphic, and extends across $C$ by Hartogs' theorem. Uniqueness
ensures that the extensions glue. The pullback of the extended form
agrees with the original section on a dense open set and hence
everywhere. A common resolution proves bimeromorphic invariance.
\end{proof}

\begin{lemma}\label{lem:exceptional}
Let $q:W\to M$ be a finite succession of blow-ups with smooth compact
centers of codimension at least two, with $M$ smooth. If $E$ is a
$q$-exceptional prime divisor and $\widehat E$ is a smooth compact
model of $E$, then
\[
P_m(\widehat E)=0\qquad(m>0).
\]
\end{lemma}

\begin{proof}
The divisor $E$ is the strict transform of an exceptional prime
divisor created at one step. If $C_j\subset M_{j-1}$ is the connected
component of the corresponding center, the divisor created at that
step is
\[
E_j=\mathbb P(N_{C_j/M_{j-1}}).
\]
Set $c_j=\codim_{M_{j-1}}C_j\ge2$. On each fiber
$P\simeq\mathbb P^{c_j-1}$ of $E_j\to C_j$, the cotangent sequence gives
\[
K_{E_j}|_P\simeq K_P\simeq\mathcal O_{\mathbb P^{c_j-1}}(-c_j).
\]
Every section of $K_{E_j}^{\otimes m}$ therefore restricts to zero on
every fiber and hence vanishes identically. The compactness assumption
on the center ensures that $E_j$ is compact. Subsequent blow-ups alter
its strict transform only bimeromorphically, so $E_j$ and
$\widehat E$ are bimeromorphic. Lemma~\ref{lem:birational-pluri}
proves the assertion.
\end{proof}

We also recall the following standard properties of the
non-isomorphism locus of a modification.
\begin{lemma}\label{lem:modification}
Let $f:W\to Z$ be a proper modification between connected smooth
complex manifolds. Then $\Exc(f)$ is the zero locus of the Jacobian
determinant of $f$, and is a union of divisors unless empty.
Its image has codimension at least two in $Z$. For every prime
divisor $H\subset Z$, there is exactly one prime divisor in $W$
dominating $H$, namely its strict transform.
\end{lemma}

\begin{proof}
A modification of a normal space is an isomorphism away from a
closed analytic subset of codimension at least two in the target.
Its fibers are connected, by Stein factorization and normality.
Where the Jacobian is nonzero, $f$ is a local biholomorphism and the
point is isolated in its fiber. Connectedness makes this fiber a
singleton. Properness allows one to shrink a neighborhood of its
image so that no other points of $W$ occur over that neighborhood;
the local inverse then shows that $f$ is an isomorphism there.
Conversely, the Jacobian is nonzero at every point of the isomorphism
locus. Since it is a nonzero holomorphic section of
$K_W\otimes f^*K_Z^{-1}$, its zero set is either empty or purely
divisorial. Its image is contained in the codimension-at-least-two
non-isomorphism locus in $Z$. Finally, over a general point of a prime
divisor $H\subset Z$, the map is an isomorphism; the closure of its
inverse image there is the strict transform of $H$ and is the unique
prime divisor dominating $H$.
\end{proof}

\section{Removing a meromorphic product structure}

\begin{proposition}\label{prop:criterion}
Let $\pi:\X\to\D$ be a proper holomorphic submersion with connected
fibers, and let $S$ be a smooth compact connected complex manifold.
Suppose that there exists a bimeromorphic map over $\D$
\[
F:\X\dashrightarrow S\times\D
\]
which is biholomorphic over $\Ds$. If $P_{m_0}(X_0)>0$ for some
$m_0>0$, then $\X\to\D$ is holomorphically trivial near $0$, and
$X_0\simeq S$.
\end{proposition}

\begin{proof}
Write $\Y=S\times\D$ and shrink $\D$ around $0$ when necessary.
Analytic elimination of indeterminacy, followed by embedded
resolution, gives a resolution of $F^{-1}:\Y\dashrightarrow\X$
from the product side by blow-ups with smooth centers; see
\cite[(1.1.1)]{Fujiki82} for analytic elimination of indeterminacy
and \cite{BierstoneMilman} for resolution and principalization in
the complex analytic category. We obtain a smooth connected
manifold $W$ and a diagram
\begin{equation}\label{eq:resolution}
q:W\longrightarrow\Y,\qquad
p:W\longrightarrow\X,\qquad F\circ p=q,
\end{equation}
where $q$ is a finite succession of blow-ups with smooth centers and
$p$ is a proper modification; both maps are over $\D$. Since
$F^{-1}$ is already an isomorphism over $\Ds$, elimination of
indeterminacy can be performed without changing that open set. Thus
all centers lie in the central fibers. The resolution sequence is
finite after shrinking around the compact central fiber. Every
intermediate space is proper over $\D$, hence the centers are compact. Blow-ups of smooth
Cartier divisors are identities and are omitted, so all nontrivial
centers have codimension at least two. After additional embedded
resolution, again supported over $0$, the strict transform $D$ of
$X_0$ under $p$ is smooth.

The divisor $D$ is the unique prime divisor of $W$ dominating $X_0$
under $p$, and $p|_D:D\to X_0$ is a modification. Thus
\begin{equation}\label{eq:positive-pluri}
P_{m_0}(D)=P_{m_0}(X_0)>0.
\end{equation}
If $D$ were $q$-exceptional, Lemma~\ref{lem:exceptional} would
contradict \eqref{eq:positive-pluri}. Hence $q(D)$ is a divisor.
As it is contained in $S\times\{0\}$, it equals that divisor.
Lemma~\ref{lem:modification} then also shows that both $p$ and $q$
are isomorphisms near a general point of $D$.

Define
\begin{equation}\label{eq:A}
A=p\bigl(\Exc(p)\cup\Exc(q)\bigr)\subset X_0.
\end{equation}
This is closed analytic by properness. Lemma~\ref{lem:modification}
gives $\codim_\X p(\Exc(p))\ge2$. If $E$ is a $q$-exceptional prime
divisor, then $E$ lies over $0$, so $p(E)\subset X_0$. If $p(E)$ were
a divisor in $\X$, it would equal $X_0$; the uniqueness assertion in
Lemma~\ref{lem:modification} would then give $E=D$, contrary to the
fact that $D$ is not $q$-exceptional. Thus $p(E)$ has codimension at
least two for every such $E$. Since the exceptional locus of the
composite of blow-ups $q$ is the union of its exceptional prime
divisors, we obtain
\begin{equation}\label{eq:codim}
\codim_\X A\ge2.
\end{equation}
Here the codimension is computed in the smooth total space $\X$.
On $\X\setminus A$, the inverse of $p$ is holomorphic and its image
avoids $\Exc(q)$. Consequently $F=q\circ p^{-1}$ is holomorphic and
locally biholomorphic there.

Pull back the product vector field by setting
\begin{equation}\label{eq:V}
V=(dF)^{-1}\left(\frac{\partial}{\partial t}\right)
\quad\text{on }\X\setminus A.
\end{equation}
In a local holomorphic frame of $T_\X$, its coefficients are
holomorphic functions. Equation~\eqref{eq:codim} and Hartogs'
extension theorem extend them across $A$; uniqueness makes the local
extensions compatible. Hence $V$ is a holomorphic vector field on
$\X$. Since $F$ preserves the base, the identity theorem gives
\begin{equation}\label{eq:horizontal}
d\pi(V)=\frac{\partial}{\partial t}
\quad\text{on all of }\X.
\end{equation}

Let $\Phi_\tau$ be the local holomorphic flow of $V$. It satisfies
\begin{equation}\label{eq:flow-base}
\pi(\Phi_\tau(x))=\pi(x)+\tau.
\end{equation}
Compactness of $X_0$ and local existence with holomorphic dependence
for ordinary differential equations give a neighborhood of $X_0$
and a common $\varepsilon>0$ on which $\Phi_\tau$ is defined for
$|\tau|<\varepsilon$. Hence
\[
H:X_0\times\D_\varepsilon\longrightarrow\X|_{\D_\varepsilon},
\qquad H(x,t)=\Phi_t(x)
\]
is defined and holomorphic. For each $t$, $H_t:X_0\to X_t$ is a
local biholomorphism. If $H_t(x)=H_t(y)$, applying the inverse flow
$\Phi_{-t}$ gives $x=y$; thus $H_t$ is injective. Its image is open
and, because $X_0$ is compact, also compact and hence closed in the
connected manifold $X_t$. It is therefore surjective. The derivative
of $H$ is invertible on vertical tangent spaces and induces the
identity on the tangent space of the disk. Thus $H$ is a bijective
local biholomorphism, hence a biholomorphism. Taking any $t\ne0$
gives $X_0\simeq X_t\simeq S$.
\end{proof}

\section{Constructing a meromorphic product by relative cycles}
\label{sec:cycles}

\subsection{Relative cycle-space inputs}

For a proper morphism $f:M\to B$, let $\Cyc_k(M/B)$ denote the
reduced relative Barlet space of nonzero compact effective
$k$-cycles whose supports are contained in a single fiber of $f$.
Thus $\Cyc_k(M/B)$ is the space denoted by $C_k(f)$ in
\cite[Proposition~12.4.3.5]{BM2}. Excluding the zero cycle makes the
map assigning to a cycle its base point unambiguous.

\begin{proposition}\label{prop:cycle-input}
Let $M\to\D$ be a proper morphism of reduced complex spaces, and let $k\ge0$. The space
$\Cyc_k(M/\D)$ has at most countably many irreducible components. It has
a holomorphic map to $\D$ and a universal analytic family of compact
cycles; the support of that family is proper over the parameter
space. The locus of reduced irreducible cycles is analytic
Zariski-open.
\end{proposition}

\begin{proof}
By \cite[Corollary~5.8.2.6, p.~99]{BM2}, the space $C_k(M)$ of
compact effective $k$-cycles is a reduced complex space representing
proper analytic families of such cycles. In particular, its
tautological family is analytic and its support is proper over
$C_k(M)$. The paragraph immediately following that corollary notes
that $C_k(M)$ is second countable and hence has at most countably
many irreducible components.

By \cite[Proposition~12.4.3.5, p.~623]{BM2}, the relative cycle
space $\Cyc_k(M/\D)$ is a closed analytic subspace of $C_k(M)$,
with its reduced structure, and the map sending a cycle supported
in $M_t$ to $t$ is holomorphic. As a subspace of a second-countable
space, it is second countable; the local finiteness of irreducible
components therefore gives the required countability. Restricting
the tautological family gives the relative universal family, whose
support remains proper over the parameter space by base change.

Finally, write a nonzero cycle as $Z=\sum_j m_j[Z_j]$, with
distinct irreducible supports $Z_j$ and positive multiplicities
$m_j$. Its weight $w(Z)=\sum_j m_j$ is upper semicontinuous for the
analytic Zariski topology by \cite[Proposition~4.7.2]{BM1}. Hence
$\{w\ge2\}$ is closed analytic. Since $Z$ is reduced and irreducible
exactly when $w(Z)=1$, the claimed open locus is its complement in
$\Cyc_k(M/\D)$.
\end{proof}

We will also use the fact that proper direct image preserves
analytic families of cycles; see \cite[Theorem~4.3.22]{BM1}.

\begin{proposition}\label{prop:relative-properness}
Let $f:M\to\D_1$ be a proper surjective Moishezon morphism of reduced
complex spaces. Fix $k\ge0$ and a
disk $0\in\D_0\Subset\D_1$. Then every irreducible component
of
\[
\Cyc_k(M_{\D_0}/\D_0),\qquad M_{\D_0}=f^{-1}(\D_0),
\]
is proper over $\D_0$. In particular, if a proper surjective
morphism is locally Moishezon near $0$, this conclusion holds on
a sufficiently small disk around $0$.
\end{proposition}

\begin{proof}
 A proper Moishezon morphism
has a proper surjective dominating model which is locally
projective over the base; see \cite[(1.5)(3)]{Fujiki82}. Thus, for
each $b\in\D_1$, there is an open neighborhood $U$ of $b$ and
morphisms
\[
q:Y\longrightarrow M_U:=f^{-1}(U),\qquad
p=f|_{M_U}\circ q:Y\longrightarrow U,
\]
where $q$ is proper and surjective and $p$ is projective. By
\cite[Corollary~12.2.3.3, p.~593]{BM2}, a projective morphism is
pre-K\"ahlerian: it admits a closed relative K\"ahler form.
Consequently these local diagrams satisfy precisely the definition
of a weakly K\"ahlerian morphism.

This property is preserved by restriction to an open subset of
the base, so $f_0:M_{\D_0}\to\D_0$ is proper and weakly
K\"ahlerian. Now \cite[Theorem~12.4.3.6, p.~623]{BM2} states that
each irreducible component of the relative space $\Cyc_k(M_{\D_0}/\D_0)$ is
proper over $\D_0$.

For the last assertion, first choose a disk $\D_1$ on which the
original morphism is Moishezon, and then choose
$0\in\D_0\Subset\D_1$. All cycle-space components in the statement
are taken after restriction to this fixed disk $\D_0$.
\end{proof}

\subsection{Openness of the graph condition}

\begin{lemma}\label{lem:graph-open}
Let $P\to T$ and $Q\to T$ be proper holomorphic submersions with
connected fibers of the same dimension, where $T$ is a complex
manifold. Let $\{Z_u\}_{u\in T}$ be an analytic family of compact
cycles in $P_u\times Q_u$. If $Z_{u_*}$ is the multiplicity-one graph
of a biholomorphism $P_{u_*}\to Q_{u_*}$, then the same holds for
$Z_u$ for all $u$ in a neighborhood of $u_*$.
\end{lemma}

\begin{proof}
We adapt the graph-openness argument of
\cite[Theorem~4.10.1]{BM1} to the relative setting. Put
$n=\dim P_u=\dim Q_u$, and let
$\mathcal Z\subset P\times_T Q$ be the reduced support of the
family, with projections $e_P$ and $e_Q$. Since $P\to T$ and
$Q\to T$ are proper, both projections and $\mathcal Z\to T$
are proper. For either projection, the locus in $\mathcal Z$
where the local fiber dimension is positive is analytic. Its image
in $T$ is closed analytic by Remmert's proper mapping theorem
\cite{Remmert}. This image does not contain $u_*$, since both
projections of $|Z_{u_*}|$ are isomorphisms. After shrinking to a
connected neighborhood of $u_*$, both $e_P$ and $e_Q$ therefore
have finite fibers and hence are finite.

By \cite[Theorem~4.3.22]{BM1}, proper direct image preserves
analytic families of cycles. Thus
\[
 Y^P_u:=(e_P)_*Z_u
\]
is an analytic family of compact $n$-cycles in $P$. Since $P_u$
is connected and smooth, it is irreducible; every irreducible
component of $Z_u$ maps finitely and surjectively onto $P_u$.
Consequently
\[
 Y^P_u=d_P(u)[P_u],\qquad d_P(u)\in\mathbb Z_{\ge0}.
\]
Choose an $n$-scale $E_P$ on $P$, using local coordinates in which
$P\to T$ is a coordinate projection, such that $E_P$ is adapted
to $[P_{u_*}]$ and has degree one on this cycle. After shrinking
the parameter neighborhood, these coordinates also give
\[
 \deg_{E_P}[P_u]=1
\]
for every nearby $u$. At $u_*$ we have
$Y^P_{u_*}=[P_{u_*}]$. Since this cycle is irreducible, the single
scale $E_P$ detects its only irreducible component. We may therefore
apply \cite[Lemma~4.8.7(i)]{BM1} to the family $Y^P_u$ in
$\mathcal C_n(P)$. After a further shrinking, $E_P$ remains adapted
and its degree is constant, so
\[
 d_P(u)=\deg_{E_P}Y^P_u
       =\deg_{E_P}Y^P_{u_*}=1.
\]
The same argument for $Q$, on a common parameter neighborhood,
gives
\[
 (e_P)_*Z_u=[P_u],\qquad (e_Q)_*Z_u=[Q_u].
\]

Write $Z_u=\sum_i m_i Z_{u,i}$. Each $Z_{u,i}$ maps finitely
and surjectively to $P_u$ with a positive integral degree, and hence
\[
 1=\sum_i m_i\deg(e_P|_{Z_{u,i}}).
\]
Thus $Z_u$ has a single irreducible component with multiplicity
one, and its projection to $P_u$ has degree one. Its projection to
$Q_u$ also has degree one. A finite bimeromorphic morphism onto a
normal complex space is an isomorphism. Both projections are
therefore isomorphisms, and $Z_u$ is the multiplicity-one graph of
a biholomorphism $P_u\to Q_u$.
\end{proof}

\subsection{Selection of an extendible family of isomorphisms}

\begin{proposition}\label{prop:graph-selection}
Let $\pi:\X\to\D$ be a proper holomorphic submersion with connected
fibers, locally Moishezon near $0$. Suppose that $S$ is a smooth
compact connected Moishezon manifold and that $X_t\simeq S$ for every
$t\ne0$. After shrinking $\D$ there is a finite base change
\[
\tau:\D_s\longrightarrow\D_t,\qquad \tau(s)=s^N,
\]
and a bimeromorphic map over $\D_s$
\begin{equation}\label{eq:selected-product}
F:\X_\tau:=\X\times_{\D_t}\D_s
\dashrightarrow S\times\D_s
\end{equation}
which is biholomorphic over $\D_s^*$.
\end{proposition}

\begin{proof}
Choose disks $0\in\D_0\Subset\D_1$ such that
$\pi|_{\X_{\D_1}}:\X_{\D_1}\to\D_1$ is Moishezon. After choosing
$\D_1$ sufficiently small, there is a proper modification
$\mu:Y\to\X_{\D_1}$ for which $Y\to\D_1$ is projective; see
\cite[(1.5)(3)]{Fujiki82}. Since $S$ is smooth and Moishezon,
there is also a proper modification $\nu:S'\to S$ with $S'$
smooth and projective, by the same result over a point. Consequently
\[
\mu\times\nu:Y\times S'\longrightarrow\X_{\D_1}\times S
\]
is a proper modification, and $Y\times S'\to\D_1$ is projective.
Thus the morphism
\[
\mathcal M_1:=\X_{\D_1}\times_{\D_1}(S\times\D_1)
\longrightarrow\D_1
\]
is Moishezon. All graph cycles below are taken in
$\mathcal M_1$, whose fixed factor is $S$ itself. By
Proposition~\ref{prop:relative-properness}, every irreducible
component of $\Cyc_n((\mathcal M_1)_{\D_0}/\D_0)$ is proper over
$\D_0$, where $n=\dim S$.

From now on replace $\D$ by $\D_0$ and restrict $\X$ accordingly;
write $\mathcal M=\X\times_\D(S\times\D)$. We select a component
of the relative cycle space only after this restriction.

For each $t\ne0$, choose an isomorphism $X_t\to S$ and denote its
graph cycle by $\Gamma_t$. The cycle space
$\Cyc_n(\mathcal M/\D)$ has countably many irreducible components.
Since $\Ds$ is uncountable, one component $C$ therefore contains
$\Gamma_t$ for uncountably many distinct values of $t$.
Proposition~\ref{prop:relative-properness} implies that
\[
h:C\longrightarrow\D
\]
is proper. Its image is a closed analytic subset by Remmert's proper mapping theorem~\cite{Remmert}
and contains those uncountably many values of $t$. A proper analytic
subset of a one-dimensional disk is discrete and hence countable.
Thus $h(C)=\D$.

Take a proper resolution $\rho:T\to C$ with $T$ smooth and
irreducible, and write $g=h\circ\rho$. Pull back the universal
cycles to $T$, and let $\mathcal Z$ be their reduced universal
support. The reduced irreducible cycle locus is a nonempty analytic
Zariski-open subset of $C$, since $C$ contains an isomorphism graph;
its inverse image in $T$ is dense. Hence $\mathcal Z$ has exactly one
irreducible component $\mathcal Z^{\mathrm{dom}}$ dominating $T$.
Indeed, two distinct dominating components would give two distinct
components in a general fiber, whereas a general pulled-back cycle
is reduced and irreducible. Discard all vertical components.

Choose a point of $T$ lying over an isomorphism graph. By
Lemma~\ref{lem:graph-open}, there is a nonempty ordinary open subset
$U\subset T$ over which all cycles are multiplicity-one isomorphism
graphs. We now verify that this fiberwise assertion gives
isomorphisms of total spaces over $U$. The two projections
\[
\mathcal Z|_U\longrightarrow\X_U:=\X\times_\D U,
\qquad
\mathcal Z|_U\longrightarrow S\times U
\]
are proper and have singleton fibers, hence are finite. Their
restrictions to each parameter fiber have degree one. On each
connected component of the targets they are therefore finite
bimeromorphic maps. The targets are smooth, hence normal, and the
source has the reduced structure, so both projections are
isomorphisms. In particular, $\mathcal Z|_U$ is smooth and coincides
with $\mathcal Z^{\mathrm{dom}}|_U$.

Resolve $\mathcal Z^{\mathrm{dom}}$ without changing its smooth
locus, hence without changing it over $U$, obtaining a smooth
irreducible space $W$ and projections
\begin{equation}\label{eq:universal-projections}
a:W\longrightarrow\X_T:=\X\times_\D T,
\qquad b:W\longrightarrow S\times T.
\end{equation}
Both maps are proper: $\mathcal Z^{\mathrm{dom}}$ is closed in
$\X_T\times_T(S\times T)$, and the two projections of this fiber
product are proper. Both targets in \eqref{eq:universal-projections}
are smooth. They are connected because their bases and fibers are
connected, and hence are irreducible. Over $U$, the maps $a$ and $b$
are isomorphisms. Their images are closed analytic subsets by
Remmert's proper mapping theorem~\cite{Remmert} and contain the nonempty open subsets $\X_U$ and
$S\times U$, respectively; irreducibility of the targets makes both
maps surjective. The source and targets have the same dimension.
For each projection the generic degree is one, since the projection
is an isomorphism over the nonempty open set lying above $U$.
Thus both maps are proper bimeromorphic morphisms, that is,
modifications. Their common structure map
$r:W\to T$ is proper.

By Lemma~\ref{lem:modification}, $\Exc(a)$ and $\Exc(b)$ are closed
analytic. Hence
\begin{equation}\label{eq:bad-parameters}
B=r\bigl(\Exc(a)\cup\Exc(b)\bigr)\subset T
\end{equation}
is closed analytic. It misses $U$, so is proper. Over $T\setminus B$
neither fiber meets either exceptional locus. Hence both maps in
\eqref{eq:universal-projections} are isomorphisms there.
The proper analytic subset $B$ therefore contains all parameters
where either resolved projection fails to be an isomorphism. This
analytic control permits the curve-selection argument below.

Because $g:T\to\D$ is proper and surjective, choose $u_0\in g^{-1}(0)$.
The set $B\cup g^{-1}(0)$ is a proper analytic subset of the smooth
irreducible manifold $T$. A generic line in a coordinate chart
through $u_0$, followed by restriction to a small disk, gives an arc
\[
\gamma:(\D_s,0)\longrightarrow(T,u_0)
\]
such that
\begin{equation}\label{eq:arc}
\gamma(s)\notin B\cup g^{-1}(0)\qquad(s\ne0).
\end{equation}
To justify this choice, take a nonzero local holomorphic function
whose zero set contains the germ of that proper analytic set. Choose
a line on which the first nonzero homogeneous term of this function
does not vanish. Its restriction to the line is not identically zero,
so, after shrinking, its only zero is the origin. This proves
\eqref{eq:arc}. In particular, $g\circ\gamma$ is nonconstant and has
an isolated zero at $0$, so
\[
g(\gamma(s))=s^N v(s),\qquad v(0)\ne0.
\]
A holomorphic $N$-th root of the unit $v$, followed by a change of
source coordinate, puts this map in the form $t=s^N$. After shrinking
the source and target disks it is a finite map $\tau:\D_s\to\D_t$.

Base-change \eqref{eq:universal-projections} by $\gamma$. Over
$\D_s^*$ the two base-changed maps are isomorphisms by
\eqref{eq:arc}. The punctured pullback is isomorphic to
$\X_\tau|_{\D_s^*}$, which is smooth and connected, and hence is
irreducible.
Let $W_\gamma$ be the reduced closure of this punctured pullback, or
equivalently the unique component dominating $\D_s$. The resulting
projections
\[
\X_\tau\xleftarrow{\ a_\gamma\ }W_\gamma
\xrightarrow{\ b_\gamma\ }S\times\D_s
\]
are proper and are isomorphisms over $\D_s^*$. Their images are
closed analytic subsets containing the dense open punctured targets,
so both maps are surjective. The targets are smooth and irreducible;
for $\X_\tau$ this follows directly from submersion coordinates for
$\pi$. Thus $a_\gamma$ and $b_\gamma$ are proper bimeromorphic
morphisms, hence modifications. The composition
$b_\gamma\circ a_\gamma^{-1}$ defines \eqref{eq:selected-product}
and is biholomorphic over $\D_s^*$.
\end{proof}

\section{Proofs of the rigidity theorems}

\begin{proof}[Proof of Theorem~\ref{thm:local-moishezon}]
Since $\kappa(S)\ge0$, choose $m_0>0$ such that $P_{m_0}(S)>0$.
Upper semicontinuity for the relative line bundle
$K_{\X/\D}^{\otimes m_0}$ gives
\begin{equation}\label{eq:semicontinuity}
P_{m_0}(X_0)\ge
\limsup_{t\to0,\ t\ne0}P_{m_0}(X_t)
=P_{m_0}(S)>0.
\end{equation}

Apply Proposition~\ref{prop:graph-selection}. For a finite base change
$\tau(s)=s^N$ it gives a bimeromorphic map
\[
F:\widehat\X:=\X\times_\D\D_s\dashrightarrow S\times\D_s
\]
which is biholomorphic over $\D_s^*$. The base-changed family
$\widehat\pi:\widehat\X\to\D_s$ is smooth and proper. In local
submersion coordinates $\pi(z,t)=t$, its fiber product has coordinates
$(z,s)$ with $t=s^N$, so even at $s=0$ it is a coordinate projection.
Moreover,
\begin{equation}\label{eq:same-center}
\widehat X_0\simeq X_0.
\end{equation}
Equations~\eqref{eq:semicontinuity} and \eqref{eq:same-center} verify
the plurigenus hypothesis in Proposition~\ref{prop:criterion}.
That proposition gives $\widehat X_0\simeq S$, hence $X_0\simeq S$.

\end{proof}

\end{document}